\pdfoutput=1
\RequirePackage{ifpdf}
\ifpdf 
\documentclass[pdftex]{sigma}
\else
\documentclass{sigma}
\fi

\usepackage[all]{xy}
\usepackage{stmaryrd}

\def\commutatif{\ar@{}[rd]|{\circlearrowleft}}

\numberwithin{equation}{section}

\newtheorem{Theorem}{Theorem}[section]
\newtheorem*{Theorem1}{Theorem~\ref{mainthm}}
\newtheorem*{Theorem2}{Theorem~\ref{mainthmAL}}

\newtheorem{Lemma}[Theorem]{Lemma}
\newtheorem{Proposition}[Theorem]{Proposition}
 { \theoremstyle{definition}
\newtheorem{Remark}[Theorem]{Remark} }

\begin{document}
\allowdisplaybreaks

\newcommand{\arXivNumber}{2105.11123}

\renewcommand{\PaperNumber}{096}

\FirstPageHeading

\ShortArticleName{Generically, Arnold--Liouville Systems Cannot be Bi-Hamiltonian}

\ArticleName{Generically, Arnold--Liouville Systems\\ Cannot be Bi-Hamiltonian}

\Author{Hassan BOUALEM~$^{\rm a}$ and Robert BROUZET~$^{\rm b}$}

\AuthorNameForHeading{H.~Boualem and R.~Brouzet}

\Address{$^{\rm a)}$~IMAG, Universit\'e de Montpellier, France}
\EmailD{\href{mailto:hassan.boualem@umontpellier.fr}{hassan.boualem@umontpellier.fr}}

\Address{$^{\rm b)}$~LAMPS, EA 4217, Universit\'e Perpignan Via Domitia, France}
\EmailD{\href{mailto:robert.brouzet@univ-perp.fr}{robert.brouzet@univ-perp.fr}}

\ArticleDates{Received May 24, 2021, in final form October 22, 2021; Published online October 29, 2021}

\Abstract{We state and prove that a certain class of smooth functions said to be BH-se\-parable is a meagre subset for the Fr\'echet topology. Because these functions are the only admissible Hamiltonians for Arnold--Liouville systems admitting a bi-Hamiltonian structure, we get that, generically, Arnold--Liouville systems cannot be bi-Hamiltonian. At the end of the paper, we determine, both as a concrete representation of our general result and as an illustrative list, which polynomial Hamiltonians $H$ of the form $H(x,y)=xy+ax^3+bx^2y+cxy^2+dy^3$ are BH-separable.}


\Classification{26A21; 26B35; 26B40; 37J35; 37J39; 58K15; 70H06}

\begin{flushright}
\begin{minipage}{90mm}
\it We dedicate this work to the memory of Pierre Molino\\ who was our master $($1935--2021$)$
\end{minipage}
\end{flushright}

\section{Introduction}

In the late 1970s and early 1980s appeared the notion of bi-Hamiltonian system in a seminal paper by F.~Magri~\cite{Mag}. This theory was widely developped a few years later by F.~Magri and C.~Morosi in \cite{MaMo}. At the same time, another fundamental work was done in this field by I.~Gel'fand and I.~Dorfman \cite{GD} but, in fact, this theory had already its roots in the so-called ``Lenard recursion formula'' as well explained in~\cite{PS}. In the late 1980s and early 1990s some works were carried out to well understand the link between, on the one hand, those of the Hamiltonian systems which are completely integrable in the sense of Arnold--Liouville and, on~the other hand, the bi-Hamiltonian systems \cite{Br, BrMoTu,Fe, Mar}. We emphasize the fact that the studies in the first three references above presented a completely different approach to the problem than the fourth. Here we focus only on the first approach which is the original approach presented by Magri and Morosi. We also point out that the present work only considers the compatibility of~symplectic forms and never examines the compatibility between degenerate Poisson structures. Moreover, other approaches have been made to provide interesting geometric structures on the phase space of an integrable Hamiltonian system. We can quote for example works of the early 2000's on separability \cite{FP,IMM} or a very recent work on Haantjes structures~\cite{TT}. The~work presented in this article is a natural questioning of the results stated in~\cite{Br, BrMoTu,Fe} leading us to examine a very particular class of Hamiltonians verifying a separability property which seems quite restrictive. Roughly speaking, such Hamiltonians need to be locally separable through a~change of local coordinates which also separates the initial (action) variables of the Hamiltonian. A function satisfying this property will be said {\it BH-separable} and we will denote by ${\cal S}_{\rm BH}$ their set. In a recent work~\cite{BouBr} we studied the topology of a class of smooth functions that we called {\it locally separable}, constituting a set denoted by~${\cal S}$. Precisely, we defined these latter functions in the following way:
a function\footnote{Actually, in this paper, all the functions we treat are considered as Hamiltonians of systems with $2n$ degrees of freedom and should therefore depend {\it a priori} on $2n$ variables. Nevertheless, we underline the fact that we are working in the framework of action-angle coordinates and thus that these Hamiltonians depend only on the $n$ action coordinates, denoted here $x_1,\dots,x_n$.} $H\colon (x_1,\dots,x_n)\mapsto H(x_1,\dots,x_n)$ defined on an open ball centered at the origin $O$ is called a {\it locally separable function} if it can be locally separated
in the sense that there exists a smooth diffeomorphism $\varphi$ (depending on $H$), fixing $O$, from an open neighborhood $V$ of $O$ onto its image $W$
\begin{gather*}
\varphi\colon\quad V\to W,\qquad
(x_1,\dots,x_n)\mapsto(u_1(x_1,\dots,x_n),\dots,u_n(x_1,\dots,x_n)),
\end{gather*}
such that for all $(x_1,\dots,x_n)\in V$ we have
\begin{gather*}
H(x_1,\dots,x_n)=H_1(u_1(x_1,\dots,x_n))+\cdots+H_n(u_n(x_1,\dots,x_n)).
\end{gather*}
The work previously mentioned explained how such separable functions are generic among smooth functions.
Now, in the case of the BH-separability studied in the present work the condition imposed on the functions is more restrictive and consists of asking that not only the function $H$ belongs to ${\cal S}$ but also that one of the smooth diffeomorphisms $\varphi$ separating $H$ also separates the coordinates $x_i$ in the sense that for all $i\in\llbracket 1,n\rrbracket$, we could write for some functions~$x_{ij}$,
\begin{gather*}
\forall i\in\llbracket 1,n\rrbracket,\qquad
x_i=\sum_{j=1}^nx_{ij}(u_j).
\end{gather*}
So, finally $H$ is BH-separable if we can find new local coordinates $(u_1,\dots,u_n)$ around $O$ such that
\begin{gather*}
\forall i\in\llbracket 1,n\rrbracket,\qquad \forall (k,l)\in\llbracket 1,n\rrbracket^2, \qquad k\not=l,\qquad \frac{\partial^2 H}{\partial u_k\partial u_l}= \frac{\partial^2 x_i}{\partial u_k\partial u_l}=0.
\end{gather*}
The aim of this paper is to prove that, this time, unlike the case of locally separable functions studied in \cite{BouBr}, the set of BH-separable functions is meagre.\footnote{We recall that a subset of a metric space is said to be {\it meagre} if it is contained in a countable union of closed subsets without interior points.} So, it will conclude that, generically, Arnold--Liouville completely integrable Hamiltonian systems cannot be bi-Hamiltonian or, in other words, that generically the complete integrability in the sense of Arnold--Liouville cannot be explained and got by the existence of an underlying bi-Hamiltonian structure.

In the first part we will recall the origin of this special separability and how it appears in the context of the bi-Hamiltonian systems. In the second part we will give our main result, stated in two different frames, namely:

\begin{Theorem1}
The set ${\cal S}_{\rm BH}$ of the BH-separable functions is a meagre subset, for the Fr\'echet topology, of the space ${\cal C}^{\infty}(U,\mathbb{R})$ of smooth real functions defined on the open ball $U=B(O,1)$ of $\mathbb{R}^n$.
\end{Theorem1}

\begin{Theorem2} Generically, an Arnold--Liouville system admits none bi-Hamiltonian structure.
\end{Theorem2}

In the third and last part, we will determine among a special class of polynomial Hamiltonians which of them are BH-separable.

\section{The mechanics origin of BH-separability}

In order to understand why we are interested in this strange property of BH-separability we have to recall briefly how it appeared in the frame of bi-Hamiltonian systems, and hence the genesis of the problem.

Let $(M,\omega,H)$ be a Hamiltonian system, namely a symplectic manifold (with a dimension denoted by $2n$) and a smooth function (the Hamiltonian) $H$ on $M$. This Hamiltonian system is said to be {\it completely integrable} if there are $n$ functions $f_1,\dots,f_n$ satisfying the following conditions \cite{AbMa, Arn,LiMa}:
\begin{enumerate}\itemsep=0pt
\item[(1)] they are first integrals of the Hamiltonian vector field $X_H$, i.e.,
\begin{gather*}
\forall i\in\llbracket 1,n\rrbracket,\qquad X_H.f_i={\rm d}f_i(X_H)=\{H,f_i\}=0,
\end{gather*}
where $\{\, ,\,\}$ denotes the Poisson bracket associated with the symplectic form $\omega$;

\item[(2)] they Poisson-commute, or are in involution, in the sense that
\begin{gather*}
\forall i,j\in\llbracket 1,n\rrbracket,\qquad
\{f_i,f_j\}=0;
\end{gather*}

\item[(3)] they are functionaly independent, at least on a dense open subset of $M$, i.e.,
\begin{gather*}
{\rm d}f_1\wedge\cdots\wedge {\rm d}f_n\not=0.
\end{gather*}
\end{enumerate}

In this situation, let us denote by $F$ the function
\begin{gather*}
F\colon\quad M\to\mathbb{R}^n,\qquad
m\mapsto(f_1(m),\dots,f_n(m)).
\end{gather*}
Because of the above property (3), $F$ is a submersion.\footnote{We suppose that the condition (3) is everywhere satisfied, replacing if necessary $M$ by the open set on which it is actually true.} If~$c$ is in the range of $F$ then $F^{-1}(c)$ is a~submanifold of $M$. If~$F$ is a proper map, then $F$ is a fibration (Ehresmann theorem \cite{Ehr}). If~we assume that its fibres are connected then they are $n$-tori $\mathbb{T}^n$. Moreover, any of these fib\-res~$F^{-1}(c)$ possesses a tubular neighborhood $\Omega$ which can be identified, up to a symplectomorphism $\Phi$, to the symplectic manifold $(U\times\mathbb{T}^n,\omega_0)$, where $U$ is some open set of $\mathbb{R}^n$ (that we can assume, without loss of generality, to be some open ball centered at the origin) endowed with coordinates $(x_1,\dots,x_n)$ and $\omega_0$ denotes the canonical symplectic form defined as
\begin{gather*}
\omega_0:=\sum_{i=1}^n{\rm d}x_i\wedge {\rm d}\theta_i,
\end{gather*}
$\theta_1,\dots,\theta_n$ being angle coordinates on the torus; this symplectic identification between $\Omega$ and $U\times\mathbb{T}^n$ can be furthermore factorized via a map $\varphi$ between the basis of the fibration $F$ and the basis $U$ of the trivial fibration $\hbox{pr}_1\colon U\times\mathbb{T}^n\to U$, in such a way that the following diagram be commutative
\begin{gather*}
\xymatrix{
	\Omega \ar[r]^{\Phi} \ar[d]_F \commutatif & U\times\mathbb{T}^n \ar[d]^{\hbox{pr}_1} \\
	F(\Omega) \ar[r]^{\varphi} & U.}
\end{gather*}
Finally, the Hamiltonian $H$ in the coordinates $(x_i,\theta_i)$ depends only on the $x_i$, so is a basic function with respect to the fibration.
All these properties constitute the statement of the classic Arnold--Liouville theorem \cite{AbMa,Arn}; the coordinates $(x_1,\dots,x_n,\theta_1,\dots,\theta_n)$ defined by the map $\Phi$ are called {\it action-angle coordinates}.
In what follows we will call {\it Arnold--Liouville system} the semi-local model described by the Arnold--Liouville theorem for a completely integrable Hamiltonian system verifying the conditions of the theorem, namely
$(U\times\mathbb{T}^n,\omega_0,H)$, where $H$ is a basic function for the trivial fibration, so a function depending only on the action coordinates~$x_i$.

In the early eighties the concept of bi-Hamiltonian system was introduced by Magri and Morosi \cite{MaMo}. Roughly speaking, the main idea of this theory is the following: if there exists a second symplectic form, compatible in a natural sense with the first one, and if the initial Hamiltonian field is also Hamiltonian with respect to this new symplectic structure, then it would be possible to generate mechanically first integrals in involution and so get the complete integrability of the Hamiltonian vector field. Precisely, two symplectic forms $\omega_0$ and $\omega_1$ on a~manifold $M$ define a $(1,1)$ tensor field $J$, called in this context a {\it recursion operator}, by the formula
\begin{gather*}
\forall X,Y\in{\cal X}(M),\qquad \omega_1(X,Y)=\omega_0(JX,Y),
\end{gather*}
where ${\cal X}(M)$ denotes the set of the vector fields on $M$.
These two symplectic forms are said {\it compatible} if the {\it Nijenhuis torsion} of $J$ \cite{Nij} vanishes, i.e.,
\begin{gather*}
\forall X,Y\in{\cal X}(M),\qquad [JX,JY]-J[JX,Y]-J[X,JY]+J^2[X,Y]=0.
\end{gather*}
This torsion appears also in the frame of K\"ahler manifolds and its vanishing is an integrability condition of the eigenspaces distribution.
Now a vector field $X$ is said bi-Hamiltonian with respect to such compatible symplectic forms if it is a Hamiltonian vector field for both forms. Actually, we generally assume that $i_X\omega_0=-{\rm d}H$ (so $X$ is Hamiltonian with respect to $\omega_0$) and that $L_X\omega_1=0$ ($\omega_1$ is invariant by $X$), condition which can also be written ${\rm d}i_X\omega_1=0$, meaning that $X$ is only locally Hamiltonian with respect to $\omega_1$.
What was proved by Magri and Morosi in~\cite{MaMo} is that, for such a bi-Hamiltonian vector field $X$, the functions $\operatorname{Tr}\big(J^k\big)$, $k\in\mathbb{N}$ constitute a~Poisson commuting family of first integrals of $X$. It produces the result that, in the case where at each point $m$ of the manifold, $J_m$ owns the maximum possible number of eigenvalues, namely~$n$, the spectrum $\{\lambda_1,\dots,\lambda_n\}$ of $J$ provides also a~Poisson commuting family of first integrals of~$X$ and so that~$X$ is completely integrable as soon as ${\rm d}\lambda_1\wedge\cdots\wedge {\rm d}\lambda_n\not=0$.
So,~roughly speaking one can say that a bi-Hamiltonian vector field is completely integrable, at least in the case where its recursion operator possesses a convenient spectrum. So the natural question is: what about the converse? If~$X$ is some completely integrable Hamiltonian system, is it possible to find a bi-Hamiltonian structure which explains this integrability, in the sense that the eigenvalues of the associated recursion operator constitute a~Poisson commuting family of first integrals for the given field? At the end of the eighties and the beginning of the nineties some works \cite{Br, BrMoTu, Fe} were done to answer this question for an Arnold--Liouville system, i.e., they study the existence of such a bi-Hamiltonian structure on the whole manifold $U\times\mathbb{T}^n$ and not only locally, this latter problem being almost obvious. The answer was negative because such an existence implies very restrictive conditions on the Hamiltonian, namely that it must be BH-separable in the sense that we have defined it above in this article.
Precisely the next result was proved \cite{Br, BrMoTu, Fe}:

\begin{Theorem}\label{BHthm}
	Let $(M_0=U\times\mathbb{T}^n,H,\omega_0)$ be an ``Arnold--Liouville system'' in the sense that $U=B(O,1)\subset\mathbb{R}^n$, endowed with $($action$)$ coordinates $x_i$, $\omega_0$ is the canonical symplectic form on~$M_0$ defined by $\omega_0=\sum_{i=1}^n{\rm d}x_i\wedge {\rm d}\theta_i$, the $\theta_i$ denoting the angle coordinates on the tori $\mathbb{T}^n$ and~$H$ is a function $($the Hamiltonian$)$ of the~$x_i$. Let us assume that:
\begin{enumerate}\itemsep=0pt	
	\item[$(i)$] This system is bi-Hamiltonian, i.e., there exists on $M_0$ a second symplectic form $\omega_1$, compatible with $\omega_0$, such that the Hamiltonian $($with respect to $\omega_0)$ vector field $X_H$ is also Hamiltonian with respect to $\omega_1$ $($at least locally$)$.
	
	\item[$(ii)$] The number of eigenvalues of the recursion operator associated to the couple $(\omega_0,\omega_1)$ is maximum $($so equal to $n)$.
	
	\item[$(iii)$] The Hamiltonian $H$ is non-degenerate in the sense that the Hessian matrix $\big(\frac{\partial^2H}{\partial x_i\partial q_j}\big)$ is invertible on a dense open subset of $U$.
\end{enumerate}
	
	Then there are around each point of $U$ local coordinates $(u_1,\dots,u_n)$ such that
\begin{gather*}
\forall (k,l)\in\llbracket 1,n\rrbracket^2,\quad
k\not=l,\quad \forall i\in\llbracket 1,n\rrbracket,\qquad
\frac{\partial^2 H}{\partial u_k\partial u_l}=\frac{\partial^2 x_i}{\partial u_k\partial u_l}=0.
\end{gather*}
	In particular the Hamiltonian $H$ is BH-separable.
\end{Theorem}

At that time, some examples were provided \cite{Br,BrMoTu,Fe} showing some Arnold--Liouville systems for which no such bi-Hamiltonian structure exists, because their Hamiltonians were not BH-separable. The aim of this paper is to state and prove that not only there exists Arnold--Liouville systems with no bi-Hamiltonian structure but also that these cases are not exceptions but the general rule.

\section{The main theorem: from function space to integrability }

\subsection{Frame and preliminary results}

In what follows we will work on the space ${\cal C}^{\infty}(U,\mathbb{R})$ of real functions defined on an open set~$U$ of~$\mathbb{R}^n$ that we can assume to be the open ball $B(O,1)$. This space will be endowed with its structure of Fr\'echet metric space \cite{Hi}.

The first statement allows us to only consider functions for which the $1$-jet at $O$ is zero.

\begin{Proposition}
	Let $H\in{\cal C}^{\infty}(B(O,1),\mathbb{R})$ and let us denote by $J_O^1(H)$ its $1$-jet at $O$. Then,
	$H\in{\cal S}_{\rm BH}$ if and only if $H-J_O^1(H)\in{\cal S}_{\rm BH}.$
	So, the BH-separability is independent of the $1$-jet of the function at the origin.
\end{Proposition}

\begin{proof}
	Indeed, $J_O^1(H)=\sum_{i=1}^na_ix_i$, where the $a_i$ are real constants so under a change of coordinates $(x_1,\dots,x_n)\mapsto(u_1,\dots,u_n)$ which separates the $x_i$, we get
\begin{gather*}
\forall (k,l)\in\llbracket 1,n\rrbracket^2, \quad k\not=l,\qquad
\frac{\partial^2 J_O^1(H)}{\partial u_k\partial u_l}
=\sum_{i=1}^na_i\frac{\partial^2 x_i}{\partial u_k\partial u_l}=0.
\end{gather*}
	It follows, because $H=\big(H-J_O^1(H)\big)+J_O^1(H)$, that
	for all	$(k,l)\in\llbracket 1,n\rrbracket^2,$ with $k\not=l$, the vanishing of the quantities
\begin{gather*}
\frac{\partial^2 H}{\partial u_k\partial u_l}\qquad \text{and}\qquad \frac{\partial^2 \big(H-J_O^1(H)\big)}{\partial u_k\partial u_l}
\end{gather*}
are equivalent.	
\end{proof}

So we can always suppose, without loss of generality, that all our functions have $O$ as a~critical point, i.e., that their $1$-jet at $O$ is zero.
From now on we will suppose that we are in this situation.

In \cite{Br,BrMoTu}, or \cite{Fe}, some examples of polynomial Hamiltonians which are not BH-separable were given. For example, as we will see in the last section, a Hamiltonian $H$ as
$H(x,y)=xy+xy^2$ is not BH-separable. However, because its Hessian is not degenerate, this function is locally separable in the sense defined in \cite{BouBr} and mentioned in the introduction; this is a direct consequence of Morse's lemma. So, in a general way, we have

\begin{Proposition}
	${\cal S}_{\rm BH}\subsetneq{\cal S}.$
\end{Proposition}

\begin{Remark}
	Of course the obvious inclusion of ${\cal S}_{\rm BH}$ into ${\cal S}$ gives
	$\mathring{\cal S}_{\rm BH}\subset\mathring{\cal S}.$
	So, if we use a recent work \cite{BouBr} which studies $\mathring{\cal S}$ we can have some informations about $\mathring{\cal S}_{\rm BH}$ Nevertheless, it~would be impossible by this mean to conclude that it is a meager subset of the space of smooth functions.
\end{Remark}

Finally, we recall the Fa\`a Di Bruno's formula \cite{AMM1, AMM2} which will furnish a precious tool in the proof of our main theorem in the next subsection:

\begin{Theorem}[Fa\`a Di Bruno's formula\footnote{This Fa\`a Di Bruno's formula is somewhat frightening at first sight! Perhaps a little illustration might clarify this. Let us consider the simple case where $H$ depends only on one variable and let us try to visualise what $P$ is, what $B$ is and so on. For the choice $k=3$, we get
\begin{gather*}
(H\circ\Phi)^{(3)}=H^{(3)}\Phi'^3+3H''\Phi'\Phi''+H'\Phi^{(3)},
\end{gather*}
			where actually all the terms in $H$, or its derivatives, in the right hand side, have to be understood as $H\circ\Phi$, $H'\circ\Phi$ and so on. Now, what is ${\cal P}_3$ and the different $P$ and $B$ here? We have
\begin{gather*}
{\cal P}_3=\{\{\{1,2,3\}\},\{\{1,2\},\{3\}\},\{\{1,3\},\{2\}\}, \{\{2,3\},\{1\},\{\{1\},\{2\},\{3\}\}\}.
\end{gather*}
			For $P=\{\{1,2,3\}\}$ we have $\vert P\vert=3$ and only one $B$ in $P$, namely $\{1,2,3\}$, so $\vert B\vert=1$. The corresponding term in the Fa\`a Di Bruno's formula is then $H^{(3)}(\Phi',\Phi',\Phi')$ (do not forget that $H^{(3)}$ is a trilinear form) and so in this context of one variable $H^{(3)}\Phi'^3$. The others terms follow in the same way.}]
	
	Let $H$ be a real function defined on an open subset~$\Omega$ of $\mathbb{R}^n$ and $\phi$ defined on an open subset $\cal O$ of $\mathbb{R}^n$ with values in $\Omega$.
	
	If $H$ and $\phi$ are functions of class ${\cal C}^k$ then for all $x=(x_1,\dots,x_n)$ in $\cal O$ we have
\begin{gather*}
(H\circ\phi)^{(k)}(x)=\sum_{P\in{\cal P}_k } H^{(\vert P\vert)}(\phi(x))\big(\phi^{(\vert B\vert)}(x)\big)_{B\in P},
\end{gather*}
	where for any smooth function $f$ of $n$ variables $f^{(k)}$ denotes its differential of order $k$, ${\cal P}_k$ denotes the set of partitions of $\llbracket 1,k\rrbracket $ and $\vert P\vert$ denotes the cardinality of $P$.
\end{Theorem}

Formulae of Fa\`a Di Bruno type, but for partial derivatives were given by M. Hardy \cite{Har} and T-W. Ma \cite{Ma}. But, for our purpose, it is not necessary to use them.

\subsection{Main theorem (function space topology version)}

Now, we can state and prove the first version of the main result of this paper, purely stated in terms of function space topology.

\begin{Theorem}\label{mainthm}	
	The set ${\cal S}_{\rm BH}$ of the BH-separable functions is a meagre subset, for the Fr\'echet topology, of the space ${\cal C}^{\infty}({\cal U},\mathbb{R})$ of smooth real functions defined on the open ball ${\cal U}=B(O,1)$ of $\mathbb{R}^n$.
	
\end{Theorem}
\begin{proof} First, let us introduce some useful notations. If $f$ is a smooth function, let us denote by~$T_k(f)$ the homogeneous part of degree $k$ of its Taylor's expansion at the origin; so $T_k(f)$ belongs to the space ${\cal H}_k^n$ of the homogeneous polynomials of degree $k$ with $n$ indeterminates $X_1,\dots,X_n$. We will denote also by $J_O^k(f)$ its $k$-jet at the origin defined as
\begin{gather*}J_O^k(f)=T_0(f)+T_1(f)+\cdots+T_k(f),\end{gather*}
	which is nothing but its Taylor polynomial at $O$ and with degree $k$.
	Now, by the very definition of the set ${\cal S}_{\rm BH}$, the fact that the function $f$ belongs to ${\cal S}_{\rm BH}$ means that there exists:
\begin{itemize}\itemsep=0pt
\item a change of variables $U \colon (x_1, \dots, x_n)\mapsto (u_1, \dots, u_n)$,
	
\item $n$ functions $g_1, \dots, g_n$ of class $C^\infty$ depending on only one variable,
	
\item $n^2$ functions $(f_{ij})_{1\leq i, j\leq n}$ of class $C^\infty$ depending only on one variable, such that, in a~neighborhood of the origin, the functions $f,x_1,\dots,x_n$ can be written as
	\begin{gather}
	f(x_1, \dots, x_n)=g_1(u_1)+\cdots+g_n(u_n),\nonumber
\\
\forall i\in\llbracket 1,n\rrbracket,\qquad x_i=\sum_{j=1}^nf_{ij}(u_j).\label{BHS}
	\end{gather}
\end{itemize}
	Let us denote respectively by $G$ and $F$ the functions
\begin{gather*}
G(u_1,\dots, u_n)=g_1(u_1)+\cdots+g_n(u_n)
\end{gather*}
and
\begin{gather*}
F(u_1,\dots, u_n)=\Bigg(\sum_{j=1}^nf_{1j}(u_j),\dots,\sum_{j=1}^nf_{nj}(u_j)\Bigg).
\end{gather*}
	By means of these functions we can write the relations (\ref{BHS}) in a condensed way, namely
\begin{gather*}
f=G\circ U\qquad \text{and}\qquad F\circ U={\rm id}_{W},
\end{gather*}
	where $W$ is an open neighborhood of the origin in $\mathbb{R}^n$.
	Differentiating the relation $F\circ U={\rm id}_{W}$ at $O$ we get that ${\rm d}F(O)\circ {\rm d}U(O)={\rm id}_{\mathbb{R}^n}$ so ${\rm d}U(O)={\rm d}F(O)^{-1}$, or with notations used in the Fa\`a di Bruno's formula, $U'(O)=F'(O)^{-1}$. Using Jacobian matrix, this relation can be written
\begin{gather*}
{\rm Jac}_O(U)=\frac{1}{\det{\rm Jac}_O(F)} \,{}^t{\rm Com}({\rm Jac}_O(F)),
\end{gather*}
	where ${\rm Com}(M)$ denotes the comatrix of a matrix $M$. It follows from this formula that each partial derivative $\frac{\partial u_i}{\partial x_p}$ is a rational fraction $\frac{P_{ip}}{Q_{ip}}$ of the numbers $f'_{rs}(0)$. The coefficients of this rational fraction are independent from $U$, $F$ and $G$; they only depend on the indices~$i$ and~$p$.
	On~the other hand,
\begin{gather*}
\frac{\partial f}{\partial x_p}(O)=\sum_{i=1}^ng_i'(0)\frac{\partial u_i}{\partial x_p}(O).
\end{gather*}
	So, let us introduce the application
\begin{gather*}
{\cal L}_1\colon\quad \mathbb{R}^n\times {\rm GL}_n(\mathbb{R})\to{\cal H}_1^n,\qquad
	(a_1,\dots,a_n,(b_{rs})_{1\leq r,s\leq n})\mapsto\sum_{p=1}^n\Bigg(\sum_{i=1}^n a_i\frac{P_{ip}}{Q_{ip}}(b_{rs}) \Bigg)X_p,
\end{gather*}
	where ${\cal H}_1^n$ denotes the space of homogeneous polynomial functions of degree $1$ with $n$ indeterminates. Then the previous discussion shows that if $f$ belongs to ${\cal S}_{\rm BH}$, then $T_1(f)$ belongs to the range of the application ${\cal L}_1$.
	
	Next, using the Fa\`a di Bruno's formula, we get the differential of order $2$ at $O$ of $U$, namely
\begin{gather*}
U''(O)(.,.)=-F'(O)^{-1}\big[F''(O)\big(F'(O)^{-1}(.),F'(O)^{-1}(.)\big)\big].
\end{gather*}
	We deduce from this relation that each of the second partial derivatives $\frac{\partial^2u_i}{\partial x_p\partial x_q}(0)$ is a rational fraction $\frac{R_{ipq}}{S_{ipq}}$ of the numbers $f'_{rs}(0)$ and $f''_{rs}(0)$, the latter not appearing in the denominators. Here again the coefficients of these fractions depend only on the indices $i$, $p$ and $q$ and not on the functions $U$, $F$ and $G$.
	
	On the other hand, we have also
\begin{gather*}
\frac{\partial^2 f}{\partial x_p\partial x_q}(O)=\sum_{i=1}^{n}g_i''(0)\frac{\partial u_i}{\partial x_p}(O)\frac{\partial u_i}{\partial x_q}(O)+g_i'(0)\frac{\partial^2 u_i}{\partial x_p\partial x_q}(O).
\end{gather*}
	
	So, let us introduce the map
\begin{gather*}
{\cal L}_2\colon\ \mathbb{R}^n\times\mathbb{R}^n\times {\rm GL}_n(\mathbb{R})\times {\cal M}_n(\mathbb{R})\to{\cal H}_2^n,
\end{gather*}
where
\begin{gather*}
{\cal L}_2(a_1, \dots, a_n, b_1, \dots, b_n, A,M):=\sum_{1\leq p,q\leq n}\Bigg(\sum_{i=1}^{n}a_i\frac{P_{ip}}{Q_{ip}}(A)\frac{P_{iq}}{Q_{iq}}(A) +b_i\frac{R_{ipq}}{S_{ipq}}(M)\Bigg)X_pX_q.
\end{gather*}
	As in the case of first derivatives, if $f$ belongs to ${\cal S}_{\rm BH}$, then $T_2(f)$ belongs to the range of the application ${\cal L}_2$.
	What happens for higher degrees of derivation? Using induction, we can easily prove that, around $O$, for all positive integer $k$, each of $k$-th order partial derivatives of the $u_i$ are rational fractions of all the derivatives of the functions $f_{rs}$. It is true for $k=1$. Indeed, we~have seen that at the origin but actually it is also true near $O$ because $F'(O)$ is invertible and so $F'(u)$ remains invertible for $u$ close to $O$. Now, if we have
\begin{gather*}
\frac{\partial^ku_i}{\partial x_{i_1}\cdots\partial{x_{i_k}}}=R\big(f'_{rs},\dots,f^{(k)}_{rs}\big),
\end{gather*}
where $R$ is some rational function with $m$ variables $t_1,\dots,t_m$, then
\begin{gather*}
\frac{\partial^{k+1}u_i}{\partial x_{i_1}\cdots\partial{x_{i_k}}\partial{x_{i_{k+1}}}}=\sum_{j=1}^m\frac{\partial R}{\partial t_j}\big(f'_{rs},\dots,f^{(k)}_{rs}\big)\frac{\partial t_j}{\partial x_{i_{k+1}}},
\end{gather*}
where, for the sake of simplicity, we denote by $t_j$ in $\frac{\partial t_j}{\partial x_{i_{k+1}}}$ something which is some $f^{(l)}_{rs}(u_s)$. So~actually,
\begin{gather*}
\frac{\partial t_j}{\partial x_{i_{k+1}}}=f^{(l+1)}_{rs}(u_s)\frac{\partial u_s}{\partial x_{i_{k+1}}}.
\end{gather*}
	Now, since the $\frac{\partial R}{\partial t_j}$ are rational fractions of the $f'_{rs},\dots,f^{(k)}_{rs}$ and the $\frac{\partial u_s}{\partial x_{i_{k+1}}}$ are also rational fractions of the $f'_{rs}$, the announced result is proved by induction.
	
	Let us return now to the derivatives of $f$. With the notations used in the Fa\`a di Bruno's formula, we have
\begin{gather*}
\frac{\partial^kf}{\partial x_{i_1}\cdots\partial{x_{i_k}}}(O)=\sum_{i=1}^n\sum_{P\in P_k}g_i^{(\vert P\vert)}(0)\prod_{B\in P}\frac{\partial^{\vert B\vert} u_i}{\prod_{j\in B}{\partial x_{i_j}}}(O).
\end{gather*}
	We can write this last formula by ordering terms with respect to the length of $P$, namely
\begin{gather*}
\frac{\partial^kf}{\partial x_{i_1}\cdots\partial{x_{i_k}}}(O)=\sum_{i=1}^n\sum_{l=1}^kg_i^{(l)}(0)\sum_{P\in P_k,\ \vert P\vert=l}\prod_{B\in P}\frac{\partial^{\vert B\vert} u_i}{\prod_{j\in B}{\partial x_{i_j}}}(O),
\end{gather*}
	the term $\sum_{P\in P_k,\ \vert P\vert=l}\prod_{B\in P}\frac{\partial^{\vert B\vert} u_i}{\prod_{j\in B}{\partial x_{i_j}}}(O)$ being a rational fraction $R_{ii_1\dots i_k}(M_1,M_2,\dots,M_k)$, where $M_j$ is the matrix $\big(f^{(j)}_{rs}(0)\big)_{r,s}$ and so $M_1$ is invertible.
	Then we can define the map
\begin{gather*}
{\cal L}_k\colon\ (\mathbb{R}^n)^k\times {\rm GL}_n(\mathbb{R})\times(M_n(\mathbb{R}))^{k-1}\to{\cal H}_n^k,
\end{gather*}
by
\begin{gather*}
{\cal L}_k\big(a^1,\dots,a^k,M_1,\dots,M_k\big)=\!\!\sum_{1\leq i_1,\dots,i_k\leq n}\!\!\Bigg(\sum_{i=1}^n\sum_{l=1}^ka_i^lR_{ii_1\dots i_k}(M_1,M_2,\dots,M_k)\Bigg)X_{i_1}\cdots X_{i_k},
\end{gather*}
in such a way that if a function $f$ belongs to ${\cal S}_{\rm BH}$ then for all positive integers $k$, $T_k(f)$ is in the range of ${\cal L}_k$.
	The dimension of the domain of the map ${\cal L}_k$ is $k\big(n+n^2\big)$ whereas those of its codomain ${\cal H}_k^n$ is
\begin{gather*}
\operatorname{dim}{\cal H}_k^n=\frac{n(n+1)\cdots(n+k-1)}{k!},
\end{gather*}
	which we can also write as
\begin{gather*}
\operatorname{dim}{\cal H}_k^n=\frac{(k+n-1)(k+n-2)\cdots (k+1)}{(n-1)!},
\end{gather*}
	showing that it is a polynomial function of $k$ with a degree $n-1$. So if $n\geq 3$, its degree is a least quadratic and so its increase at infinity ensures that for $k$ sufficiently large it will be strictly larger than the linear expression $k\big(n+n^2\big)$ and so the inequality
	\begin{gather}\label{ineq}
	k\big(n+n^2\big)<\operatorname{dim}{\cal H}_k^n=\frac{n(n+1)\cdots(n+k-1)}{k!}
	\end{gather}
	holds.
	On the contrary, for $n=2$, the two compared dimensions are respectively $6k$ and $k+1$ and so the inequality $6k<k+1$ is always false. For that reason we will have to deal separately with the cases $n\geq 3$ and $n=2$.
	Let us begin with the case $n\geq 3$.
	Let $k$ be an integer such that the inequality (\ref{ineq}) holds. Then the range of ${\cal L}_k$ has a Lebesgue measure equal to zero. We know that $\mathbb{R}^{kn}\times\mathbb{R}^{kn^2}$ is a countable union of compact sets $(K_i)_{i\geq0}$, so the range of ${\cal L}_k$ is a~countable union of the compact sets $({\cal L}_k(K_i))_{i\geq0}$ and each ${\cal L}_k(K_i)$ has no interior points since its measure is zero. Because $T_k$ is an open and continuous map we get that each $T_k^{-1}({\cal L}_k(K_i))$ is a closed set with no interior points. So we deduce that ${\cal S}_{\rm BH}$ is meagre.
	
	It remains to examine the case $n=2$. In this case, we can proceed as previously but by replacing $T_k(f)$ by the $k$-jet of $f$, namely $J^k_O(f)$. This time we introduce a map
\begin{gather*}
\widehat{\cal L}_k\colon\ \mathbb{R}^{2k}\times {\rm GL}_2(\mathbb{R})\times\mathbb{R}^{4(k-1)}\times\mathbb{R}\to\mathbb{R}_k[x,y],
\end{gather*}
	defined by
\begin{gather*}
\widehat{\cal L}_k(a,b,c,d)=d+\sum_{i=1}^{k}{\cal L}_i(a,b,c),
\end{gather*}
	where $\mathbb{R}_k[x,y]$ denotes the space of homogeneous polynomial functions with two indeterminates and total degree $\leq k$.
	With these notations the expression
\begin{gather*}
\widehat{\cal L}_k\left(\big(g_i^{(j)}\big)_{(i,j)\in[\![1,n]\!]\times[\![1,k]\!]},F^{(l)}(0))_{l\in[\![1,k]\!]},\lambda\right)
\end{gather*}
	is nothing else than the $k$-jet of $f$. The function $f$ belongs to ${\cal S}_{\rm BH}$ then $J^k_O(f)$ is in the range of $\widehat{\cal L}_k$.
	We know that $\mathbb{R}_k[x,y]$ has dimension equal to $\frac{(k+2)(k+1)}{2}$. For $k>9$ we have $6k+1<\frac{(k+1)(k+2)}{2}$, so the range of $\widehat{\cal L}_k$ has its measure equal to zero. Thus, we can apply the same arguments as in the case $n\geq3$ and conclude that ${\cal S}_{\rm BH}$ is a meagre set.
\end{proof}

\subsection{Main theorem (integrability version): Arnold--Liouville {\it\bfseries vs} BH-systems}

Before we state and prove our main result, we need to introduce some notations and special subsets of the space of smooth functions and to state a lemma.
For $H\in{\cal C}^\infty(B(O,1))$ let us define the open set $\Omega_H$ by
\begin{gather*}
\Omega_H:=\big\{x\in B(O,1),\, \operatorname{rank}(\operatorname{Hess}_x(H))=n\big\}.
\end{gather*}
Let us also introduce the two sets
\begin{gather*}
{\cal ND}:=\big\{H\in {\cal C}^\infty(B(O,1)),\ B(O,1)\subset\overline{\Omega_H}\big\}\qquad \text{and}\qquad {\cal D}={\cal C}^\infty(B(O,1))\setminus{\cal ND}.
\end{gather*}
A function of the first set is said to be {\it non-degenerate}; it is exactly the condition assumed on the Hamiltonians of Theorem \ref{BHthm}. A function of the second set, which we will say {\it degenerate}, has a Hessian matrix with a vanishing determinant on at least a small open ball $B(a,r)$ centered on a point $a\in B(O,1)$.

\begin{Lemma}
	The set ${\cal D}$ is meagre.
\end{Lemma}

\begin{proof}
	Let $({\cal O}_k)_{k\in\mathbb{N}}$ be a countable neighborhood basis of $B(O,1)$. Then, clearly,
\begin{gather*}
{\cal D}=\bigcup_{k\in\mathbb{N}}{\cal D}_k,
\end{gather*}
	where ${\cal D}_k$ is the set of functions $H$ such that $\det\operatorname{Hess}H=0$ everywhere on ${\cal O}_k$. Each of the sets~${\cal D}_k$ is closed. Let us prove that $\mathring{\cal D}_k=\varnothing$. For that, let us fix some
	$H\in{\cal D}_k$ and some point~$a$ in~${\cal O}_k$.
	Arbitrarily close to $\operatorname{Hess}_a(H)$ we can find a symmetric matrix $A$ with a rank equal to~$n$. Let us define the function $f$ on $B(O,1)$ by
\begin{gather*}
f(x)=H(x)+{}^txAx-{}^tx\operatorname{Hess}_a(H)x,
\end{gather*}
	where $x$ denotes here the column vector of the $x_i$.
	This function $f$ can be arbitrarily close to $H$ according as to the choice made on $A$ and does not belong to ${\cal D}_k$, so $H\not\in\mathring{\cal D}_k$. So ${\cal D}$ is a meagre set as a countable union of nowhere dense closed sets.
\end{proof}

According to Theorems \ref{BHthm} and \ref{mainthm} we can state that:

\begin{Theorem}\label{mainthmAL}
	Generically, an Arnold--Liouville system admits none bi-Hamiltonian structure in the sense of Theorem~$\ref{BHthm}$.
\end{Theorem}

By this statement we understand the following result: the set ${\cal H}_{\rm BH}$ of Hamiltonians $H\in{\cal C}^{\infty}(B(O,1),\mathbb{R})$ for which the Arnold--Liouville system $(M_0,\omega_0,H)$ admits a bi-Hamiltonian structure (with the conditions of Theorem \ref{BHthm}) is meagre for the Fr\'echet's topology.

\begin{proof}
	Using the notation of the previous lemma we can write
\begin{gather*}
{\cal H}_{\rm BH}=\big({\cal D}\cap{\cal H}_{\rm BH}\big)\cup\big({\cal ND}\cap{\cal H}_{\rm BH}\big).
\end{gather*}
	According to Theorem \ref{BHthm}, we have the inclusion
	$\left({\cal ND}\cap{\cal H}_{\rm BH}\right)\subset{\cal S}_{\rm BH},$
	so Theorem \ref{mainthm} implies that ${\cal ND}\cap{\cal H}_{\rm BH}$ is meagre.
	But ${\cal D}\cap{\cal H}_{\rm BH}\subset{\cal D}$ so, using the previous lemma, we get that ${\cal D}\cap{\cal H}_{\rm BH}$ is meagre. It results in that ${\cal H}_{\rm BH}$ is meagre as the union of two meagre sets.
\end{proof}

\section{BH-separability of a special class of polynomial Hamiltonians}

\subsection{Preliminary considerations}

The previous sections have shown that among the smooth functions, the BH-separable ones are rather ``rare''. Nevertheless, of course, they do exist. In this last part, we will determine among a class of particular Hamiltonians with two variables,\footnote{These two variables will be denoted here $x$, $y$ and can be interpreted as the two action coordinates of a~Hamiltonian function written in action-angle coordinates.} which are BH-separable. This study will be doubly useful: first, obviously, to obtain these rare candidates but also to visualize perfectly and concretely the previously studied phenomenon of scarcity in a restricted family of functions.

A challenging study would be to determine (in the case of two variables) which Hamiltonians of the form $H=Q+C$, where $Q$ is a non-degenerate quadratic form and $C$ a general cubic polynomial, are BH-separable.
Let us begin with a few remarks.

$\bullet$ The quadratic part $Q$ has a signature $(2,0)$, $(1,1)$ or $(0,2)$ because it is assumed to be non-degenerate. So, using a linear change of coordinates it can be written $x^2+y^2$, $x^2-y^2$ or $-x^2-y^2$. Unfortunately, under a change of coordinates, even under a {\it linear} change of coordinates, the property of BH-separability does not necessarily remain. Indeed, the problem is that the coordinate change must separate both functions $x$ and $y$. So, the only study of these three cases would not allow us to cover all the situations.

$\bullet$ The study of the three cases $Q=x^2+y^2$, $Q=x^2-y^2$, $Q=-x^2-y^2$ (the third would be essentially the same as the first) is tedious and leads to awful calculations, even using some symbolic computing tool. For all those reasons we will present here only one of these cases and we have chosen the case where $Q$ has signature $(1,1)$. Moreover, we will deal with $Q=xy$ instead of $Q=x^2-y^2$, the first case being a little more pleasant than the last.
So we will search, inside the family
\begin{gather*}
xy+ax^3+bx^2y+cxy^2+dy^3
\end{gather*}
indexed by $(a,b,c,d)\in\mathbb{R}^4$,
the $BH$-separable elements. We can remark that this family is a~$4$-dimensional affine subspace of the space of smooth functions.

We can also remark that all these Hamiltonians are locally separable in the sense studied in~\cite{BouBr}, because they are Morse's functions at the origin.\footnote{We mean that the origin is a critical point of the function and that its Hessian matrix at this point is invertible.}

$\bullet$ If $H$ is BH-separable
there is some change of variable $(x,y)\mapsto (u,v)$ around the origin~$O$, fixing $O$, and which separates $H$, $x$ and $y$. In other words we have six smooth functions $f$, $g$, $h$, $k$, $H_1$, $H_2$, each of them depending only on one variable ($u$ or $v$) such that
\begin{gather*}
x=f(u)+g(v),\qquad y=h(u)+k(v)
\end{gather*}
and
\begin{gather}
H(f(u)+g(v),h(u)+k(v))=H_1(u)+H_2(v)\tag{$*$}.
\end{gather}
Replacing if necessary $f$ by $f-f(0)$ we can assume that $f(0)=0$. Since the change of variable fixes $O$, then $g(0)=0$. In the same way we can suppose that $h(0)=k(0)=0$. Because the considered map is a change of variable, then necessarily $(f'(0),g'(0))\not=(0,0)$ and $(h'(0),k'(0))\not=(0,0)$.
Because in these families of Hamiltonians the $1$-jet at $(0,0)$ is zero, it~remains zero in coordinates $u$, $v$ and the quadratic part (the Hessian) is tensorial, so its signature is invariant under any change of coordinates, in particular in coordinates $u$, $v$. So, writing
\begin{gather*}
H_1(u)=\alpha u^2+o\big(u^2\big),\qquad H_2(v)=\beta v^2+o\big(v^2\big)
\end{gather*}
and comparing the two expressions of $H$ in coordinates $u$, $v$ given by $(*)$ we get
\begin{gather*}
(a_1u+a_2v)(a_3u+a_4v)=\alpha u^2+\beta v^2,
\end{gather*}
where
$\left(\begin{smallmatrix}
a_1&a_2\cr a_3& a_4
\end{smallmatrix}\right)$
denotes the Jacobian matrix at the origin of the map $(u,v)\mapsto(x,y)$.
Because the signature of the considered quadratic form is $(1,1)$, then necessarily $\alpha\not=0$ and $\beta\not=0$ and because $\alpha=a_1a_3$ and $\beta=a_2a_4$, none of the real numbers $a_1$, $a_2$, $a_3$, $a_4$ is zero. So, without loss of generality, we can assume that $x=u+v$ (or $y=u+v$). Indeed, because $a_1=f'(0)\not=0$, the map $u\mapsto U=f(u)$ is a local diffeomorphism around $0$ fixing $0$ and we can write $u=f^{-1}(U)$; in~the same way, because $a_2=g'(0)\not=0$, we get locally $v=g^{-1}(V)$ and so we can write $x=U+V$ and $y=h\big(f^{-1}(U)\big)+k\big(g^{-1}(V)\big)$. The same thing could be done with the functions~$h$ and~$k$ instead of $f$ and $g$ in order to get $y=U+V$ if we would wish it.\footnote{If we chose to deal with $Q=x^2-y^2$, or with $Q=x^2+y^2$ the situation would be quite different. Indeed, we~get
\begin{gather*}
(a_1u+a_2v)^2\pm(a_3u+a_4v)^2=\alpha u^2+\beta v^2,
\end{gather*}
	a condition which does not imply that the four coefficients $a_1$, $a_2$, $a_3$, $a_4$ are not zero. For example $a_2$ may be zero but, in this case, $a_1\not=0$ and $a_4\not=0$ since $a_1a_4-a_2a_3\not=0$. It follows that in this case we can always assume that $x=u+q(v)$ and $y=p(u)+v$ or $x=p(u)+v$ and $y=u+q(v)$.}

\subsection{Some families of BH-separable Hamiltonians}

The previous discussion allows us to assume that if a suitable change of coordinates exists then we can choose $x=u+v$ and $y=p(u)+q(v)$ or $y=u+v$ and $x=p(u)+q(v)$. We will deal with the first case then we will exchange the roles of $x$ and $y$ to get all the cases. The condition of BH-separability requires that
\begin{gather}\label{Huv}
\frac{\partial^2}{\partial u\partial v}\big(H(u+v,p(u)+q(v))\big)=0.
\end{gather}
Writing the relation (\ref{Huv}) for $H=xy+ax^3+bx^2y+cxy^2+dy^3$ we get
\begin{gather}
6a(u+v)+2b(p+q)+(p'+q')(1+2b(u+v)+2c(p+q))\nonumber
\\ \qquad
{}+p'q'(6d(p+q)+2c(u+v))=0.\label{Huvfor}
\end{gather}
Evaluating the relation (\ref{Huvfor})
at $u=v=0$, we get $q'(0)=-p'(0)$. Let us denote $k=p'(0)$.
Now, evaluating separately (\ref{Huvfor}) at $v=0$, then at $u=0$, we get the two differential equations
\begin{gather*}
6au+2bp+(p'-k)(1+2bu+2cp)-kp'(6dp+2cu)=0
\end{gather*}
and
\begin{gather*}
6av+2bq+(q'+k)(1+2bv+2cq)+kq'(6dq+2cv)=0.
\end{gather*}
These two ODE are equivalent by the change $k\mapsto-k$.
Rewriting the first one we get
\begin{gather*}
p'(1+2bu+2cp-6kdp-2kcu)+6au+2bp-k(1+2bu+2cp)=0
\end{gather*}
or, denoting $A=2c-6kd$, $B=2(b-kc)$ and $C=6a-2kb$,
\begin{gather*}
p'(Ap+Bu+1)+Bp+Cu-k=0.
\end{gather*}
Let us suppose that $A\not=0$ and let us denote $L=Ap+Bu+1$. Then the last differential equation can be written as
\begin{gather*}
L'L+\big(AC-B^2\big)u-B-Ak=0.
\end{gather*}
Taking into account that $L(0)=1$, we deduce that
$\frac{L^2}{2}+\big(AC-B^2\big)\frac{u^2}2-(B+Ak)u-\frac12=0$ and so
that
$L^2=-\big(AC-B^2\big)u^2+2(B+Ak)u+1$, or
$L=\sqrt{-\big(AC-B^2\big)u^2+2(B+Ak)u+1}$, in~order to finally get
\begin{gather}\label{pu}
p(u)=\frac{-1-Bu+\sqrt{-\big(AC-B^2\big)u^2+2(B+Ak)u+1}}{A}.
\end{gather}

We get the formula for $q$ replacing $k$ by $-k$:
\begin{gather}\label{qv}
q(v)=\frac{-1-B'v+\sqrt{-\big(A'C'-B'^2\big)v^2+2(B'-A'k)v+1}}{A'},
\end{gather}
where $A'=2c+6kd$, $B'=2(b+kc)$ and $C'=6a+2kb$.

So, the previous formulae (\ref{pu}) and (\ref{qv}) give the necessary form of a change of variables $(u,v)\mapsto(x=u+v,\,y=p(u)+q(v))$ which could separate the Hamiltonian $H$ in turn. Let us denote
\begin{gather*}
\widetilde{H}(u,v):=H(u+v,p(u)+q(v)),
\end{gather*}
the functions $p$ and $q$ being replaced by the expressions obtained in (\ref{pu}) and (\ref{qv}) and $\widetilde{H}_{u^kv^l}$ the derivatives $\frac{\partial^{k+l}\widetilde{H}}{\partial u^k\partial v^l}$.
All the terms $\widetilde{H}_{u^kv^l}(0,0)$ have to be zero. Using some symbolic computing tool\footnote{For instance Maple.} we verify that
\begin{gather*}
\widetilde{H}_{uv}(0,0)=\widetilde{H}_{u^2v}(0,0)=\widetilde{H}_{uv^2}(0,0)=0.
\end{gather*}
But higher derivatives do not vanish. Indeed we get successively:
\begin{gather}\label{eq1}
\widetilde{H}_{u^2v^2}(0,0)=-24\big(cdk^4+(3ad-bc)k^2+ab\big),
\\
\widetilde{H}_{u^3v^2}(0,0)=24\big(3dk^3-ck^2-bk+3a\big)\nonumber
\\ \hphantom{\widetilde{H}_{u^3v^2}(0,0)=}
{}\times\big({-}6cdk^3-\big(c^2+3bd\big)k^2+3(bc-3ad)k+3ac+b^2\big)
\label{eq2}
\end{gather}
and
\begin{gather}\label{eq3}
\widetilde{H}_{u^3v^3}(0,0)=-288c\big(3dk^3-ck^2-bk+3a\big)\big({-}3dk^2+b\big)\big({-}3dk^3-ck^2+bk+3a\big).
\end{gather}
This last equality is especially interesting and we will base the next discussion on it by distinguishing several cases. We emphasize that all of the following calculations were performed using a symbolic calculation tool.

\subsection*{First case: $\boldsymbol {c=0}$}

Let us first notice that in this case we must assume $d\not=0$ because if $d=0$, expressions of the functions $p$ and $q$ have no sense (denominators vanish). So the case $c=d=0$ will have to be studied separately.

In this case, the condition (\ref{eq3}) is, of course, satisfied and the two others ((\ref{eq1}) and (\ref{eq2})) give respectively
\begin{gather}\label{relc0}
a\big(3dk^2+b\big)=0\qquad \text{and}\qquad \big(3dk^3-bk+3a\big)\big({-}3bdk^2-9adk+b^2\big)=0.
\end{gather}
These two conditions invite us to look at the special case $a=0$. If $a=0$ then the first of the previous relations is satisfied and the second one gives, because $k\not=0$, $b\big(3dk^2-b\big)^2=0$. Consequently, the next subcases have to be considered:
\begin{itemize}\itemsep=0pt

\item $a=b=0$: $H(x,y)=xy+dy^3$.
In this case the functions $p$ and $q$ become
\begin{gather*}
p(u)=\frac{-1+\sqrt{1-12dk^2u}}{6dk},\qquad q(v)=\frac{-1+\sqrt{1-12dk^2v}}{6dk},
\end{gather*}
and they allow us to separate $H$ since
\begin{gather*}
\widetilde{H}(u,v)=\frac{u\sqrt{1-12dk^2u}}{9dk} +\frac{\sqrt{1-12dk^2u}}{54d^2k^3}-\frac{v\sqrt{1-12dk^2v}}{9dk}-\frac{\sqrt{1-12dk^2v}}{54d^2k^3}.
\end{gather*}

\item $a=0$ and $b=3dk^2$: $H=xy+3dk^2xy^2+dy^3$.
In this case we get
\begin{gather*}
p(u)=-ku,\qquad q(v)=kv
\end{gather*}
and
\begin{gather*}
\widetilde{H}(u,v)=-ku^2-4k^3du^3+kv^2+4k^3dv^3.
\end{gather*}
So $H$ is BH-separable.
\end{itemize}

Now, if $a\not=0$, necessarily $b=-3dk^2$. In this case the second part of the relation (\ref{relc0}) writes $27kd\big(2dk^3+a\big)\big(2dk^3-a\big)=0$. Because $k$ and $d$ are not zero, we get two possibilities: $a=\pm2dk^3$. Replacing these values in relation~(\ref{eq2}) we do not get zero so the case $a\not=0$ cannot lead to a~case of BH-separability when $c=0$.

\subsection*{Second case: $\boldsymbol{b=3dk^2}$}

Under this assumption, condition (\ref{eq1}) becomes
\begin{gather*}
48dk^2\big(3a-ck^2\big)=0.
\end{gather*}
This condition leads us to consider the two subcases $d=0$ and $3a-ck^2=0$.

Under the assumption $d=0$, condition (\ref{eq2}) becomes
\begin{gather*}
24c\big(3a-ck^2\big)^2=0,
\end{gather*}
and so the relation $3a-ck^2=0$ must be verified since we have already noted that $c$ and $d$ cannot be simultaneously equal to zero in the present discussion (this particular case will be discussed separately). So, let us only assume:
\begin{itemize}\itemsep=0pt
\item $3a-ck^2=0$: $H=xy+\frac{ck^2}{3}x^3+3dk^2x^2y+cxy^2+dy^3$.
In this situation, we get
\begin{gather*}
p(u)=-ku,\qquad q(k)=kv,
\end{gather*}
and
\begin{gather*}
\widetilde{H}=-ku^2+kv^2+\frac{4}{3}ck^2u^3+\frac{4}{3}ck^2v^3-4dk^3u^3+4dk^3v^3,
\end{gather*}
and so $H$ is BH-separable.
\end{itemize}

\subsection*{Third case: $\boldsymbol{3dk^3-ck^2-bk+3a=0}$}

In this case, in order to find a solution for the function $p(u)$ such that $p(0)=0$ it is necessary\footnote{Indeed, without any additional condition, the symbolic calculus tool we used does not indicate any solution for this Cauchy problem. Then, leaving the initial condition, we obtain the general solution depending on an arbitrary constant. A simple calculation then gives that if the solution $p$ satisfies $p(0)=0$ then necessarily $c+3dk=0$. Conversely, under this assumption, we get $p(u)=-ku$.} that $c+3dk=0$ and, under this supplementary condition we find $p(u)=-ku$ and $q(v)=kv$. Calculation of $\widetilde{H}$ gives
\begin{gather*}
\widetilde{H}=-ku^2+kv^2-6k^3du^3-6k^3duv^2-4k^3dv^3-\frac{2}{3}bku^3+2bkuv^2+\frac{4}{3}bkv^3,
\end{gather*}
which is separate in $u$, $v$ if and only if $b-3dk^2=0$, giving finally
\begin{gather*}
\widetilde{H}=-ku^2+kv^2-8k^3du^3.
\end{gather*}
In this case, the assumption of all the relations leads to
$a=-dk^3$, $b=3dk^2$, $c=-3dk$, so
\begin{gather*}
H=xy-dk^3x^3+3dk^2x^2y-3dkxy^2+dy^3.
\end{gather*}

\subsection*{Fourth case: $\boldsymbol{-3dk^3-ck^2+bk+3a=0}$}

It is the same situation as the previous one replacing $k$ by $-k$.

Partial conclusion: the previous discussion and calculations gave four families of BH-separable Hamiltonians:
\begin{gather*}
H=xy+dy^3,
\\
H=xy+3dk^2x^2y+dy^3,
\\
H=xy+\frac{ck^2}{3}x^3+3dk^2x^2y+cxy^2+dy^3,
\\
H=xy-dk^3x^3+3dk^2x^2y-3dkxy^2+dy^3.
\end{gather*}

The second one is a special case of the third one with $c=0$ and the fourth one is also a~special case of the third one with $c=-3dk$. So we only need to consider two families:
\begin{gather*}
H=xy+dy^3\qquad \text{and}\qquad H=xy+\frac{ck^2}{3}x^3+3dk^2x^2y+cxy^2+dy^3.
\end{gather*}
We can notice that the first one seems to be a particular case of the second one corresponding to $k=0$. But we have to remind ourselves that here the parameter $k$ comes from the linear part of the change of coordinates, namely $k=p'(0)$ and so does not become equal to zero. By the way, the reader could check that if we choose $k=0$ and $cd\not=0$ in the second family we do not obtain a BH-separable Hamiltonian. This fact clearly shows that we must necessarily avoid $k=0$ in the second family.
So, for the moment we have two families of BH-separable Hamiltonians, the first one is a $1$-parameter family and the second one, a $3$-parameters family, namely (taking $3c$ instead of $c$ in the second one):
\begin{gather}
H=xy+dy^3,\qquad d\in\mathbb{R},\nonumber
\\
H=xy+ck^2x^3+3dk^2x^2y+3cxy^2+dy^3,\qquad (c,d,k)\in\mathbb{R}^2\times\mathbb{R}^*.\label{Hxy}
\end{gather}
It remains to study the case which we did not deal with: $c=d=0$.
In this case, we find that necessarily
\begin{gather*}
p(u)=\frac{-ku-3au^2-bu^2k}{1+2bu},\qquad q(v)=\frac{kv-3av^2+bv^2k}{1+2bv}.
\end{gather*}
Replacing them in $H$ we get that the term in $u^2v^2$ is equal to $-6ab$, so necessarily $a=0$ or $b=0$. For example, if $a=0$, we get that the term in $u^2v^3$ is equal to $2kb^3$ and so vanishes if and only if $b=0$. If we suppose $b=0$, we get in the same manner that $a=0$. So, once we have $c=d=0$, right away we get $a=b=0$ and so $H=xy$, which is a trivial case already contained in the first family for $d=0$.
The last thing that we have to consider is that we have prioritized the variable $x$ over the variable $y$ and so we have to exchange $x$ and $y$ to get forgotten Hamiltonians, namely
\begin{gather}
H=xy+dx^3,\qquad d\in\mathbb{R},\nonumber
\\
H=xy+dx^3+3cx^2y+3dk^2xy^2+ck^2y^3,\qquad (c,d,k)\in\mathbb{R}^2\times\mathbb{R}^*.\label{Hyx}
\end{gather}
The two forms (\ref{Hxy}) and (\ref{Hyx}) are not equivalent because, for example, $H=xy+y^3$ belongs to one of the families (\ref{Hxy}) and not to one of the families (\ref{Hyx}). We can summarize all this discussion by this result:

\begin{Theorem}
	Among Hamiltonians of the form
\begin{gather*}
H=xy+ax^3+bx^2y+cxy^2+dy^3,\qquad (a,b,c,d)\in\mathbb{R}^4,
\end{gather*}
	the only ones to be BH-separable are those belonging to one of the four families:
\begin{gather*}	
H_d=xy+dy^3,\qquad d\in\mathbb{R},
\\	
H_d=xy+dx^3,\qquad d\in\mathbb{R},
\\	
H_{c,d,k}=xy+ck^2x^3+3dk^2x^2y+3cxy^2+dy^3,\qquad (c,d,k)\in\mathbb{R}^2\times\mathbb{R}^*,
\\	
H_{c,d,k}=xy+dx^3+3cx^2y+3dk^2xy^2+ck^2y^3,\qquad (c,d,k)\in\mathbb{R}^2\times\mathbb{R}^*.	
\end{gather*}
\end{Theorem}

\begin{Remark}
	We can outline that for the two $1$-parameter families the change of coordinates allowing to separate $H$, $x$ and $y$ is not linear whereas for the two others families with $3$ parameters, it is linear.
\end{Remark}

\begin{Remark}
	This result is a very good illustration of the main theorem of this paper, namely Theorem~\ref{mainthm} stating that Hamiltonians which are not BH-separable are generic. Indeed the family of Hamiltonians studied above is a $4$-dimensional affine space whereas, in that space, the set ${\cal S}_{\rm BH}$ of convenient (BH-separable) functions consists only in $d$-parameters families with $d<4$. Precisely, the smooth maps
\begin{gather*}
(c,d,k)\mapsto\big(ck^2,3dk^2,3c,d\big)\qquad \text{and}\qquad (c,d,k)\mapsto\big(d, 3c,3dk^2,ck^2\big)
\end{gather*}
	have clearly a Jacobian matrix with a rank $r\in\{2,3\}$ and so the set ${\cal S}_{\rm BH}$ is a finite union of regular hypersurfaces and $2$-codimensional submanifolds, so something with a Lebesgue measure equal to zero and which is a meagre subset of the ambient affine space.
	
\end{Remark}

\subsection*{Acknowledgements}

We thank Timothy Neal for his proofreading and improving the english language of our paper. We also thank Roman G. Smirnov for drawing our attention to the importance of Lenard's early work in the genesis of the theory of bi-Hamiltonian systems. Finally, we thank the anonymous referees for their insightful comments and careful reading which greatly improved this article.

\pdfbookmark[1]{References}{ref}
\LastPageEnding

\end{document}